\documentclass[reqno,12pt]{amsart}
\usepackage{amssymb}
\usepackage{amsfonts}
\usepackage{amsmath}
\usepackage{graphicx}
\usepackage{amscd,amsthm}

\newtheorem{theorem}{Theorem}
\theoremstyle{plain}
\newtheorem{acknowledgement}{Acknowledgement}

\newtheorem{corollary}{Corollary}

\newtheorem{definition}{Definition}
\newtheorem{example}{Example}

\newtheorem{proposition}{Proposition}
\newtheorem{remark}{Remark}

\numberwithin{equation}{section}

\input{tcilatex}

\begin{document}
\author{}
\title{}
\maketitle

\begin{center}
\thispagestyle{empty} \pagestyle{myheadings} 
\markboth{\bf Yilmaz Simsek
}{\bf Generating function for the Bernstein polynomials and its applications}

\textbf{{\Large Generating functions for the Bernstein polynomials: A
unified approach to deriving identities for the \textbf{Bernstein basis
functions}}}

\bigskip

\textbf{Yilmaz Simsek}\\[0pt]

Department of Mathematics, Faculty of Science University of Akdeniz TR-07058
Antalya, Turkey \\[0pt]

E-mail\textbf{: }ysimsek@akdeniz.edu.tr\\[0pt]

\bigskip

\textbf{{\large {Abstract}}}\medskip
\end{center}

\begin{quotation}
The main aim of this paper is to provide a unified approach to deriving
identities for the Bernstein polynomials using a novel generating function.
We derive various functional equations and differential equations using this
generating function. Using these equations, we give new proofs both for a
recursive definition of the Bernstein basis functions and for derivatives of
the $n$th degree Bernstein polynomials. We also find some new identities and
properties for the Bernstein basis functions. Furthermore, we discuss
analytic representations for the generalized Bernstein polynomials through
the binomial or Newton distribution and Poisson distribution with mean and
variance. Using this novel generating function, we also derive an identity
which represents a pointwise orthogonality relation for the Bernstein basis
functions. Finally, by using the mean and the variance, we generalize
Szasz-Mirakjan type basis functions.
\end{quotation}

\bigskip


\bigskip


\section{Introduction and main definition}

The Bernstein polynomials have many applications in approximations of
functions, in statistics, in numerical analysis, in $p$-adic analysis and in
the solution of differential equations. It is also well-known that in
Computer Aided Geometric Design polynomials are often expressed in terms of
the Bernstein basis functions.

Many of the known identities for the Bernstein basis functions are currently
derived in an \textit{ad hoc} fashion, using either the binomial theorem,
the binomial distribution, tricky algebraic manipulations or blossoming. The
main purpose of this work is to construct novel generating functions for the
Bernstein polynomials. Using these novel generating functions, we develop a
unify approach both to standard and to new identities for the Bernstein
polynomials.

The following definition gives us generating functions for the Bernstein
basis functions:

\begin{definition}
Let $a$ and $b$ be nonnegative real parameters with $a\neq b$. Let $m$ a be
positive integer and let $x\in \left[ a,b\right] $. Then the Bernstein basis
functions $\mathbb{Y}_{k}^{n}(x;a,b,m)$\ are defined by means of the
following generating function:%
\begin{eqnarray}
f_{\mathbb{Y},k}(x,t;a,b,m) &=&\sum_{j=0}^{\infty }\sum_{l=0}^{k}\left( 
\begin{array}{c}
j+m-1 \\ 
j%
\end{array}%
\right) (-1)^{k-l}\frac{t^{k}x^{l}a^{j+k-l}b^{-m-j}e^{(b-x)t}}{l!(k-l)!}
\label{1Be} \\
&=&\sum_{n=0}^{\infty }\mathbb{Y}_{k}^{n}(x;a,b,m)\frac{t^{n}}{n!},  \notag
\end{eqnarray}%
where $t\in \mathbb{C}$ and $0^{j}=\left\{ 
\begin{array}{cc}
0 & \text{if }j\neq 0, \\ 
1 & \text{if }j=0.%
\end{array}%
\right. $
\end{definition}

The remainder of this study is organized as follows:

Section 2: We find many \textit{functional equations} and \textit{%
differential equations} of this novel generating function. Using these
equations, many properties of the Bernstein basis functions can be
determined. For instance, we give new proofs of the recursive definition of
the Bernstein basis functions as well as a novel derivation for the two term
formula for the derivatives of the $n$th degree Bernstein basis functions.
We also prove many other properties of the Bernstein basis functions via
functional equations.

Jetter and St\"{o}ckler \cite{Jetter2003} proved an identity for
multivariate Bernstein polynomials on a simplex, which is considered a
pointwise orthogonality relation. The integral version of this identity
provides a new representation for the polynomial basis dual to the Bernstein
basis. An identity for the reproducing kernel is used to define
quasi-interpolants of arbitrary order. As an application of the identity of
Jetter and St\"{o}ckler, Abel and Li \cite{abel} gave Proposition \ref{Prob.
1}, in Section 3. Their method is based on generating functions, which
reveals the general structure of the identity. As an applications of
Proposition \ref{Prob. 1} they derive generating functions for the Baskakov
basis functions and the Szasz-Mirakjan basis functions. Using Eq-(\ref%
{1BB1Subg}) in Section 2, they exhibit a special case of the identity of
Jetter and St\"{o}ckler for the Bernstein basis functions. In Section 3; we
give relations between the Bernstein basis functions, the binomial
distribution and the Poisson distribution. Using the Poisson distribution,
we give generating functions for the Szasz-Mirakjan type basis functions. By
using Abel and Li's \cite{abel} method, and applying our generating
functions to Proposition \ref{Prob. 1}, we derive identities which give
pointwise orthogonality relations for the Bernstein polynomials and the
Szasz-Mirakjan type basis functions.

\section{Unified approach to deriving new proofs of the identities and
properties for the Bernstein polynomials}

The Bernstein polynomials and related polynomials\ have been studied and
defined in many different ways, for examples by $q$-series, complex
functions, $p$-adic Volkenborn integrals and many algorithms.

In this section, we provide fundamental properties of the Bernstein basis
functions and their generating functions. We introduce some functional
equations and differential equations of the novel generating functions for
the Bernstein basis functions. We also give new proofs of some well known
properties of the Bernstein basis functions via functional equations and
differential equations.

\subsection{Generating Functions}

We now modify (\ref{1Be}) as follows:

By the negative binomial theorem, we have%
\begin{equation}
\frac{1}{b^{m}(1-\frac{a}{b})^{m}}=\frac{1}{b^{m}}\sum_{j=0}^{\infty }\left( 
\begin{array}{c}
j+m-1 \\ 
j%
\end{array}%
\right) a^{j}b^{-m-j}.  \label{1BB1}
\end{equation}%
Substituting (\ref{1BB1}) into (\ref{1Be}), we get%
\begin{eqnarray*}
f_{\mathbb{Y},k}(x,t;a,b,m) &=&\frac{t^{k}e^{(b-x)t}}{(b-a)^{m}k!}%
\sum_{l=0}^{k}\left( 
\begin{array}{c}
k \\ 
l%
\end{array}%
\right) (-1)^{k-l}x^{l}a^{k-l} \\
&=&\sum_{n=0}^{\infty }\mathbb{Y}_{k}^{n}(x;a,b,m)\frac{t^{n}}{n!}.
\end{eqnarray*}%
Thus we obtain the following novel generating function, which is a
modification of (\ref{1Be}):%
\begin{eqnarray}
f_{\mathbb{Y},k}(x,t;a,b,m) &=&\frac{t^{k}\left( x-a\right) ^{k}e^{(b-x)t}}{%
(b-a)^{m}k!}  \label{1BB3} \\
&=&\sum_{n=0}^{\infty }\mathbb{Y}_{k}^{n}(x;a,b,m)\frac{t^{n}}{n!}.  \notag
\end{eqnarray}

\begin{remark}
If we set $a=0$ and $b=1$ in (\ref{1BB3}), we obtain a result given by
Simsek and Acikgoz \cite{Simsek Acikgoz} and Acikgoz and Arici \cite%
{AcikgozSerkan}:%
\begin{equation*}
\frac{(xt)^{k}}{k!}e^{(1-x)t}=\sum_{n=0}^{\infty }B_{k}^{n}(x)\frac{t^{n}}{n!%
},
\end{equation*}%
so that, obviously;%
\begin{equation*}
\mathbb{Y}_{k}^{n}(x;0,1,m)=B_{k}^{n}(x),
\end{equation*}%
where $B_{k}^{n}(x)$ denote the Bernstein polynomials.
\end{remark}

By using the Taylor series for $e^{(b-x)t}$\ in (\ref{1BB3}), we get%
\begin{equation*}
\frac{(x-a)^{k}}{(b-a)^{m}k!}\sum_{n=0}^{\infty }\mathbb{(}b-x\mathbb{)}^{n}%
\frac{t^{n+k}}{n!}=\sum_{n=0}^{\infty }\mathbb{Y}_{k}^{n}(x;a,b,m)\frac{t^{n}%
}{n!}.
\end{equation*}%
Comparing the coefficients of $t^{k}$ on the both sides of the above
equation, we arrive at the following theorem:

\begin{theorem}
\label{BeT-1}Let $a$ and $b$ be nonnegative real parameters with $a\neq b$.
Let $m$ be a positive integer and let $x\in \left[ a,b\right] $. Let $k$ and 
$n$ be non-negative integers with $n\geq k$. Then%
\begin{equation}
\mathbb{Y}_{k}^{n}(x;a,b,m)=\left( 
\begin{array}{c}
n \\ 
k%
\end{array}%
\right) \frac{\left( x-a\right) ^{k}(b-x)^{n-k}}{(b-a)^{m}},  \label{3BE}
\end{equation}%
where $k=0$, $1$,$\cdots $, $n$, and $\left( 
\begin{array}{c}
n \\ 
k%
\end{array}%
\right) =\frac{n!}{k!(n-k)!}$.
\end{theorem}

\begin{remark}
For $m=n$, the Bernstein basis functions of degree $n$ are defined by (\ref%
{3BE}).
\end{remark}

\begin{remark}
In the special case when $m=n$, Theorem \ref{BeT-1} immediately yields the
corresponding well known results concerning the Bernstein basis functions $%
B_{k}^{n}(x)$ that appears for example in Goldman \cite[p. 384, Eq.(24.6)]%
{GoldmanBOOK} and cf. \cite{bernstein}:%
\begin{equation*}
\mathbb{Y}_{k}^{n}(x;a,b,n)=B_{k}^{n}(x;a,b)=\left( 
\begin{array}{c}
n \\ 
k%
\end{array}%
\right) \frac{\left( x-a\right) ^{k}(b-x)^{n-k}}{(b-a)^{n}},
\end{equation*}%
where $k=0$, $1$,$\cdots $, $n$ and $x\in \lbrack a,b]$. One can easily see
that%
\begin{equation}
B_{k}^{n}(x)=\left( 
\begin{array}{c}
n \\ 
k%
\end{array}%
\right) x^{k}(1-x)^{n-k},  \label{1BB3a}
\end{equation}%
where $k=0,1,\cdots ,n$ and $x\in \lbrack 0,1]$ cf. \cite{abel}-\cite{Simsek
Acikgoz}. In \cite{GoldmanBOOK}, Goldman gives many properties of the
Bernstein polynomials $B_{k}^{n}(x,a,b)$. The functions $B_{0}^{n}(x,a,b),%
\cdots ,B_{n}^{n}(x,a,b)$ are called the Bernstein basis functions. Goldman 
\cite{GoldmanBOOK}, in Chapter 26, shows that the Bernstein basis functions
form a basis for the polynomials of degree $n$. The \textbf{Bezier curve} $%
B(t)$\ with control points $P_{0}$,$\cdots $, $P_{n}$ is defined as follows:%
\begin{equation*}
B(t)=\sum_{k=0}^{n}P_{k}B_{k}^{n}(x,a,b)\text{ cf. \cite{GoldmanBOOK}.}
\end{equation*}
\end{remark}

\begin{remark}
By using (\ref{3BE}), we have%
\begin{equation*}
\sum_{n=0}^{\infty }\mathbb{Y}_{k}^{n}(x;a,b,m)\frac{t^{n}}{n!}%
=\sum_{n=0}^{\infty }\left( 
\begin{array}{c}
n \\ 
k%
\end{array}%
\right) \frac{\left( x-a\right) ^{k}(b-x)^{n-k}}{(b-a)^{m}}\frac{t^{n}}{n!}.
\end{equation*}%
From this equation, we obtain%
\begin{equation*}
\sum_{n=0}^{\infty }\mathbb{Y}_{k}^{n}(x;a,b,m)\frac{t^{n}}{n!}=\frac{\left(
x-a\right) ^{k}t^{k}}{k!(b-a)^{m}}\sum_{n=k}^{\infty }(b-x)^{n-k}\frac{%
t^{n-k}}{\left( n-k\right) !}.
\end{equation*}%
The\ series on the right hand side is the Taylor series for $e^{(b-x)t}$;
thus we arrive at (\ref{1BB3}).
\end{remark}

Substituting $m=n$ in (\ref{3BE}), we now give \textit{another well-known
generating function} for the Bernstein basis functions:%
\begin{equation*}
\sum_{n=0}^{\infty }\left( \sum_{k=0}^{n}\mathbb{Y}_{k}^{n}(x;a,b,n)t^{k}%
\right) \frac{z^{n}}{n!}=\sum_{n=0}^{\infty }\left( \sum_{k=0}^{n}\left( 
\begin{array}{c}
n \\ 
k%
\end{array}%
\right) t^{k}\left( \frac{x-a}{b-a}\right) ^{k}\left( \frac{b-x}{b-a}\right)
^{n-k}\right) \frac{z^{n}}{n!}.
\end{equation*}%
By using the Cauchy product in the above equation, we have%
\begin{equation*}
\sum_{n=0}^{\infty }\left( \sum_{k=0}^{n}\mathbb{Y}_{k}^{n}(x;a,b,n)t^{k}%
\right) \frac{z^{n}}{n!}=\sum_{n=0}^{\infty }\left( t\frac{x-a}{b-a}\right) 
\frac{z^{n}}{n!}\sum_{n=0}^{\infty }\left( \frac{b-x}{b-a}\right) \frac{z^{n}%
}{n!}.
\end{equation*}%
From this equation, we find that%
\begin{equation*}
\sum_{n=0}^{\infty }\left( \sum_{k=0}^{n}\mathbb{Y}_{k}^{n}(x;a,b,n)t^{k}%
\right) \frac{z^{n}}{n!}=e^{z\left( \frac{b-x}{b-a}+t\frac{x-a}{b-a}\right)
}.
\end{equation*}%
After some elementary calculations in the above relation, we arrive at the
following\textit{\ generating function} for the Bernstein basis functions:%
\begin{equation}
\sum_{k=0}^{n}\mathbb{Y}_{k}^{n}(x;a,b,n)t^{k}=\left( \frac{b-x}{b-a}+t\frac{%
x-a}{b-a}\right) ^{n}.  \label{1Gold0}
\end{equation}

\begin{remark}
If we set $a=0$, $b=1$ and $m=n$ in (\ref{1Gold0}), then we have%
\begin{equation}
\sum_{k=0}^{n}B_{k}^{n}(x)t^{k}=\left( \left( 1-x\right) +tx\right) ^{n}.
\label{1BB1Subg}
\end{equation}%
This generating functions is given by Goldman \cite{GoldmanBook3}-\cite[%
Chapter 5, pages 299-306]{GoldmanBook2}. Goldman \cite{GoldmanBook3}-\cite[%
Chapter 5, pages 299-306]{GoldmanBook2} also constructs the following
generating functions the univariate and bivariate Bernstein basis functions:%
\begin{equation*}
\sum_{k=0}^{n}B_{k}^{n}(x)e^{ky}=\left( \left( 1-x\right) +te^{y}\right)
^{n},
\end{equation*}%
\begin{equation*}
\sum_{i+j+k=n}B_{i,j,k}^{n}(s,t)x^{i}y^{j}=\left( \left( 1-s-t\right)
+sx+ty\right) ^{n},
\end{equation*}%
where%
\begin{equation*}
B_{i,j,k}^{n}(s,t)=\left( 
\begin{array}{c}
n \\ 
ijk%
\end{array}%
\right) s^{i}t^{j}\left( 1-s-t\right) ^{k}\text{ and }\left( 
\begin{array}{c}
n \\ 
ijk%
\end{array}%
\right) =\frac{n!}{i!j!k!}
\end{equation*}%
and%
\begin{equation*}
\sum_{i+j+k=n}B_{i,j,k}^{n}(s,t)e^{ix}e^{jy}=\left( \left( 1-s-t\right)
+se^{x}+te^{y}\right) ^{n}.
\end{equation*}
\end{remark}

\textbf{Below are some well-known properties of the Bernstein basis functions%
}:

\textit{Non-negative property}:%
\begin{equation}
\mathbb{Y}_{k}^{n}(x;a,b,m)\geq 0\text{, for }0\leq a\leq x\leq b.
\label{1Gold1}
\end{equation}

\textit{Symmetry property}:%
\begin{equation}
\mathbb{Y}_{k}^{n}(x;a,b,m)=\mathbb{Y}_{n-k}^{n}(b+a-x;a,b,m).
\label{1Gold2}
\end{equation}

\textit{Corner values}:%
\begin{equation}
\mathbb{Y}_{k}^{n}(a;a,b,m)=\left\{ 
\begin{array}{cc}
0 & \text{if }k\neq 0, \\ 
1 & \text{if }k=0,%
\end{array}%
\right.  \label{1Gold3}
\end{equation}%
and%
\begin{equation}
\mathbb{Y}_{k}^{n}(b;a,b,m)=\left\{ 
\begin{array}{cc}
0 & \text{if }k\neq n, \\ 
1 & \text{if }k=n.%
\end{array}%
\right.  \label{1Gold4}
\end{equation}

\textit{Alternating sum}:

Substituting $m=n$ in (\ref{3BE}), we get%
\begin{equation*}
\sum_{n=0}^{\infty }\left( \sum_{k=0}^{n}(-1)^{k}\mathbb{Y}%
_{k}^{n}(x;a,b,n)\right) \frac{t^{n}}{n!}=\sum_{n=0}^{\infty }\left(
\sum_{k=0}^{n}\frac{\left( \frac{a-x}{b-a}\right) ^{k}\left( \frac{b-x}{b-a}%
\right) ^{n-k}}{k!(n-k)!}\right) t^{n}.
\end{equation*}%
By using the Cauchy product in the above equation, we have%
\begin{equation*}
\sum_{n=0}^{\infty }\left( \sum_{k=0}^{n}(-1)^{k}\mathbb{Y}%
_{k}^{n}(x;a,b,n)\right) \frac{t^{n}}{n!}=e^{\left( \frac{a+b-2x}{b-a}%
\right) t}.
\end{equation*}%
From this relation, we arrive at the following formula\ for the alternating
sum\textit{.}%
\begin{equation}
\sum_{k=0}^{n}(-1)^{k}\mathbb{Y}_{k}^{n}(x;a,b,n)=\left( \frac{a+b-2x}{b-a}%
\right) ^{n}.  \label{1Gold5}
\end{equation}

\begin{remark}
If we set $a=0$, $b=1$ and $m=n$, then Eq-(\ref{1Gold1})-Eq-(\ref{1Gold5})
reduce to Goldman's results \cite{GoldmanBook3}-\cite[Chapter 5, pages
299-306]{GoldmanBook2}. In \cite{GoldmanBook3} and \cite[Chapter 5, pages
299-306]{GoldmanBook2}, Goldman also gives many identities and properties
for the univariate and bivariate Bernstein basis functions, for example
boundary values, maximum values, partitions of unity, representation of
monomials, representation in terms of monomials, conversion to monomial
form, linear independence, Descartes' law of sign, discrete convolution,
unimodality, subdivision, directional derivatives, integrals, Marsden
identities, De Boor-Fix formulas, and the other properties.
\end{remark}

A \textit{Bernstein polynomial} $\mathcal{P}(x,a,b,m)$ is a polynomial
represented in the Bernstein basis functions:%
\begin{equation}
\mathcal{P}(x,a,b,m)=\sum_{k=0}^{n}c_{k}^{n}\mathbb{Y}_{k}^{n}(x;a,b,m).
\label{1BB2}
\end{equation}

\begin{remark}
If we set $a=0$, $b=1$ and $m=n$ (\ref{1BB2}), then we have%
\begin{equation*}
P(x)=\sum_{k=0}^{n}c_{k}^{n}B_{k}^{n}(x)
\end{equation*}%
cf. \cite{Goldman}.
\end{remark}

By using (\ref{1BB3}), we obtain the following \textit{functional equation:}%
\begin{equation*}
f_{\mathbb{Y},k_{1}}(x,t;a,b,m_{1})f_{\mathbb{Y},k_{2}}(x,t;a,b,m_{2})=\frac{%
\left( 
\begin{array}{c}
k_{1}+k_{2} \\ 
k_{1}%
\end{array}%
\right) }{2^{k_{1}+k_{2}}}f_{\mathbb{Y},k_{1}+k_{2}}(x,2t;a,b,m_{1}+m_{2}),
\end{equation*}%
where%
\begin{equation*}
\left( 
\begin{array}{c}
k_{1}+k_{2} \\ 
k_{1}%
\end{array}%
\right) =\left( 
\begin{array}{c}
k_{1}+k_{2} \\ 
k_{2}%
\end{array}%
\right) =\frac{\left( k_{1}+k_{2}\right) !}{k_{1}!k_{2}!}.
\end{equation*}

By using the definition of the novel generating function $f_{\mathbb{Y}%
,k}(x,t;a,b,m)$ in the preceding equation, we get%
\begin{eqnarray*}
&&\sum_{n=0}^{\infty }\mathbb{Y}_{k_{1}}^{n}(x;a,b,m_{1})\frac{t^{n}}{n!}%
\sum_{n=0}^{\infty }\mathbb{Y}_{k_{2}}^{n}(x;a,b,m_{2})\frac{t^{n}}{n!} \\
&=&\sum_{n=0}^{\infty }\mathbb{Y}_{k_{1}+k_{2}}^{n}(x;a,b,m_{1}+m_{2})\frac{%
2^{n-k_{1}-k_{2}}\left( k_{1}+k_{2}\right) !t^{n}}{n!k_{1}!k_{2}!}.
\end{eqnarray*}%
And using the Cauchy product in this equation, we have%
\begin{eqnarray*}
&&\sum_{n=0}^{\infty }\left( \dsum\limits_{j=0}^{n}\left( 
\begin{array}{c}
n \\ 
j%
\end{array}%
\right) \mathbb{Y}_{k_{1}}^{j}(x;a,b,m_{1})\mathbb{Y}%
_{k_{2}}^{n-j}(x;a,b,m_{2})\right) \frac{t^{n}}{n!} \\
&=&\sum_{n=0}^{\infty }\mathbb{Y}_{k_{1}+k_{2}}^{n}(x;a,b,m_{1}+m_{2})\frac{%
2^{n-k_{1}-k_{2}}\left( k_{1}+k_{2}\right) !t^{n}}{n!k_{1}!k_{2}!}.
\end{eqnarray*}%
Comparing the coefficients of $\frac{t^{n}}{n!}$ on the both sides of the
above equation, we arrive at the following theorem:

\begin{theorem}
Let $m_{1}$ and $m_{2}$ be integers. Then the following identity holds:%
\begin{equation*}
\mathbb{Y}_{k_{1}+k_{2}}^{n}(x;a,b,m_{1}+m_{2})=\frac{%
2^{k_{1}+k_{2}-n}k_{1}!k_{2}!}{\left( k_{1}+k_{2}\right) !}%
\dsum\limits_{j=0}^{n}\left( 
\begin{array}{c}
n \\ 
j%
\end{array}%
\right) \mathbb{Y}_{k_{1}}^{j}(x;a,b,m_{1})\mathbb{Y}%
_{k_{2}}^{n-j}(x;a,b,m_{2}).
\end{equation*}
\end{theorem}

Observe that if we set $a=0$ and $b=1$, then we have%
\begin{equation*}
B_{k_{1}+k_{2}}^{n}(x)=\frac{2^{k_{1}+k_{2}-n}k_{1}!k_{2}!}{\left(
k_{1}+k_{2}\right) !}\dsum\limits_{j=0}^{n}\left( 
\begin{array}{c}
n \\ 
j%
\end{array}%
\right) B_{k_{1}}^{j}(x)B_{k_{2}}^{n-j}(x).
\end{equation*}

Note that many new identities can be found via functional equations for the
novel generating functions of the Bernstein basis functions. We derive some
functional equations and identities related to the generating functions and
the Bernstein basis functions in the remainder of this section.

\subsection{Subdivision property}

The following functional equation of the novel generating functions is
fundamental to driving the subdivision property for the Bernstein basis
functions.

Let we us define%
\begin{equation}
f_{\mathbb{Y},j}(xy,t;a,b,n)=f_{\mathbb{Y},j}\left( x,t\left( \frac{y-a}{b-a}%
\right) ;a,b,n\right) e^{t\left( \frac{b-y}{b-a}\right) }.  \label{1BB1Sub}
\end{equation}%
From this generating function, we have the following theorem:

\begin{theorem}
\label{Theorem3}Let $a\leq yx\leq b$. Then the following identity holds: 
\begin{equation*}
\mathbb{Y}_{j}^{n}(xy;a,b,n)=\dsum\limits_{k=j}^{n}\mathbb{Y}%
_{j}^{k}(x;a,b,k)\mathbb{Y}_{k}^{n}(y;a,b,n-k).
\end{equation*}
\end{theorem}

\begin{proof}
By equations (\ref{1BB3}) and (\ref{1BB1Sub}), we obtain%
\begin{eqnarray*}
&&\sum_{n=j}^{\infty }\mathbb{Y}_{j}^{n}(xy;a,b,n)\frac{t^{n}}{n!} \\
&=&\left( \sum_{n=0}^{\infty }\mathbb{Y}_{j}^{n}(x;a,b,n)\left( \frac{y-a}{%
b-a}\right) ^{n}\frac{t^{n}}{n!}\right) \left( \sum_{n=0}^{\infty }\frac{%
\left( \frac{b-y}{b-a}\right) ^{n}t^{n}}{n!}\right) .
\end{eqnarray*}%
Using the Cauchy product in this equation, we get%
\begin{equation*}
\sum_{n=j}^{\infty }\mathbb{Y}_{j}^{n}(xy;a,b,m)\frac{t^{n}}{n!}%
=\sum_{n=j}^{\infty }\left( \dsum\limits_{k=j}^{n}\mathbb{Y}_{j}^{n}(x;a,b,k)%
\frac{\left( \frac{y-a}{b-a}\right) ^{k}\left( \frac{b-y}{b-a}\right) ^{n-k}%
}{k!\left( n-k\right) !}\right) t^{n}.
\end{equation*}%
Substituting (\ref{3BE}) into the above equation then after some elementary
manipulations, we arrive at the desired result.
\end{proof}

\begin{remark}
Substituting $a=0$, $b=1$ and $m=n$ into Theorem \ref{Theorem3}, we have%
\begin{equation}
B_{j}^{n}(xy)=\dsum\limits_{k=j}^{n}B_{j}^{k}(x)B_{k}^{n}(y).
\label{1BB2Sub}
\end{equation}%
The above identity is essentially the subdivision property for the Bernstein
basis functions. This identity is a bit tricky to prove with algebraic
manipulations.
\end{remark}

\begin{remark}
Goldman \cite{GoldmanBook3}-\cite[Chapter 5, pages 299-306]{GoldmanBook2}
proves equation (\ref{1BB2Sub}) with algebraic manipulations. He also proves
the following subdivision properties:%
\begin{equation*}
B_{j}^{n}(\left( 1-y\right)
x+y)=\dsum\limits_{k=0}^{j}B_{j-k}^{n-k}(x)B_{k}^{n}(y),
\end{equation*}%
and%
\begin{equation*}
B_{j}^{n}(\left( 1-y\right) x+yz)=\dsum\limits_{k=0}^{n}\left(
\dsum\limits_{p+q=j}B_{p}^{n-k}(x)B_{q}^{k}(z)\right) B_{k}^{n}(y)
\end{equation*}%
for the others see cf. \cite{GoldmanBook3}-\cite[Chapter 5, pages 299-306]%
{GoldmanBook2}.
\end{remark}

\subsection{Differentiating the generating function}

In this section we give higher order derivatives of the Bernstein basis
functions by differentiating the generating function in (\ref{1BB3}) with
respect to $x$. Using Leibnitz's formula for the $l$th derivative, with
respect to $x$, of the product $f_{\mathbb{Y},k}(x,t;a,b,m)$ of two
functions $g(t,x;a,b)=\frac{t^{k}\left( x-a\right) ^{k}}{(b-a)^{m}k!}$ with $%
a\neq b$ and $h(t,x;b)=e^{(b-x)t}$, we obtain the following \textit{higher
order partial derivative equation}:%
\begin{equation*}
\frac{\partial ^{l}f_{\mathbb{Y},k}(x,t;a,b,m)}{\partial x^{l}}%
=\dsum\limits_{j=0}^{l}\left( 
\begin{array}{c}
l \\ 
j%
\end{array}%
\right) \left( \frac{\partial ^{j}g(t,x;a,b)}{\partial x^{j}}\right) \left( 
\frac{\partial ^{l-j}h(t,x;b)}{\partial x^{l-j}}\right) .
\end{equation*}%
From this equation, we arrive at the following theorem:

\begin{theorem}
\label{TeoX}Let $l$ be a non-negative integer. Then%
\begin{equation*}
\frac{\partial ^{l}f_{\mathbb{Y},k}(x,t;a,b,m)}{\partial x^{l}}%
=\dsum\limits_{j=0}^{l}\left( 
\begin{array}{c}
l \\ 
j%
\end{array}%
\right) (-1)^{l-j}\frac{t^{l}}{(b-a)^{j}}f_{\mathbb{Y},k-j}(x,t;a,b,m-j).
\end{equation*}
\end{theorem}

By using Theorem \ref{TeoX}, we obtain higher order derivatives of the
Bernstein basis functions by the following theorem:

\begin{theorem}
\label{TeoX1}Let $a$ and $b$ be nonnegative real parameters with $a\neq b$.
Let $m$ be a positive integer and let $x\in \left[ a,b\right] $. Let $k$, $l$
and $n$ be nonnegative integers with $n\geq k$. Then%
\begin{equation*}
\frac{d^{l}\mathbb{Y}_{k}^{n}(x;a,b,m)}{dx^{l}}=\dsum%
\limits_{j=0}^{l}(-1)^{l-j}\left( 
\begin{array}{c}
n \\ 
n-l,l-j,j%
\end{array}%
\right) \frac{l!}{(b-a)^{j}}\mathbb{Y}_{k-j}^{n-l}(x;a,b,m-j),
\end{equation*}%
where%
\begin{equation*}
\left( 
\begin{array}{c}
n \\ 
x,y,z%
\end{array}%
\right) =\frac{n!}{x!y!z!}\text{, with }n=x+y+z.
\end{equation*}
\end{theorem}

\begin{remark}
Substituting $a=0$, $b=1$ and $m=n$ into Theorem \ref{TeoX1}, we have%
\begin{equation*}
\frac{d^{l}B_{k}^{n}(x)}{dx^{l}}=\dsum\limits_{j=0}^{l}(-1)^{l-j}\left( 
\begin{array}{c}
n \\ 
n-l,l-j,j%
\end{array}%
\right) l!B_{k-j}^{n-l}(x),
\end{equation*}%
or%
\begin{equation*}
\frac{d^{l}B_{k}^{n}(x)}{dx^{l}}=\frac{n!}{(n-l)!}\dsum%
\limits_{j=0}^{l}(-1)^{l-j}\left( 
\begin{array}{c}
n \\ 
j%
\end{array}%
\right) B_{k-j}^{n-l}(x),
\end{equation*}
cf. (\cite{GoldmanBook3}, \cite[Chapter 5, pages 299-306]{GoldmanBook2}).
\end{remark}

Substituting $l=1$ into Theorem \ref{TeoX1}, we arrive at the following
corollary:

\begin{corollary}
\label{BeT-2}Let $a$ and $b$ be nonnegative real parameters with $a\neq b$.
Let $m$ be a positive integer and let $x\in \left[ a,b\right] $. Let $k$ and 
$n$ be nonnegative integers with $n\geq k$. Then%
\begin{equation*}
\frac{d}{dx}\mathbb{Y}_{k}^{n}(x;a,b,m)=n\left( \frac{\mathbb{Y}%
_{k-1}^{n-1}(x;a,b,m-1)-\mathbb{Y}_{k}^{n-1}(x;a,b,m-1)}{b-a}\right) .
\end{equation*}
\end{corollary}

\begin{remark}
By setting $m=n$ in Corollary \ref{BeT-2}, we arrive at the known known
result recorded by Goldman \cite{GoldmanBOOK}:%
\begin{equation*}
\frac{d}{dx}B_{k}^{n}(x;a,b)=n\left( \frac{%
B_{k-1}^{n-1}(x;a,b)-B_{k}^{n-1}(x;a,b)}{b-a}\right) .
\end{equation*}
\end{remark}

\begin{remark}
One can also see the following special case of Theorem \ref{BeT-2} when $a=0$
and $b=1$:%
\begin{equation*}
\frac{d}{dx}B_{k}^{n}(x)=n\left( B_{k-1}^{n-1}(x)-B_{k}^{n-1}(x)\right)
\end{equation*}%
cf. \cite{abel}-\cite{Simsek Acikgoz}.
\end{remark}

\subsection{\textbf{Recurrence Relation}}

In this section by using higher order derivatives of the novel generating
function with respect to $t$, we derive a partial differential equation.
Using this equation, we shall give a new proof of the recurrence relation
for the Bernstein basis functions.

Differentiating Eq-(\ref{1Be}) with respect to $t$, we prove a recurrence
relation for the polynomials $\mathbb{Y}_{k}^{n}(x;a,b,m)$. This recurrence
relation can also be obtained from Eq-(\ref{3BE}). By using Leibnitz's
formula for the $v$th derivative, with respect to $t$, of the product $f_{%
\mathbb{Y},k}(x,t;a,b,m)$ of two function $g(t,x;a,b)=\frac{t^{k}\left(
x-a\right) ^{k}}{(b-a)^{m}k!}$ with $a\neq b$\ and $h(t,x;b)=e^{(b-x)t}$, we
obtain another higher order partial differential equation as follows:%
\begin{equation*}
\frac{\partial ^{v}f_{\mathbb{Y},k}(x,t;a,b,m)}{\partial t^{v}}%
=\dsum\limits_{j=0}^{v}\left( 
\begin{array}{c}
v \\ 
j%
\end{array}%
\right) \left( \frac{\partial ^{j}g(t,x;a,b)}{\partial t^{j}}\right) \left( 
\frac{\partial ^{v-j}h(t,x;b)}{\partial t^{v-j}}\right) .
\end{equation*}%
From the above equation, we have the following theorem:

\begin{theorem}
\label{TheoremRECUr}Let $v$ be an integer number. Then%
\begin{equation*}
\frac{\partial ^{v}f_{\mathbb{Y},k}(x,t;a,b,m)}{\partial t^{v}}%
=\dsum\limits_{j=0}^{v}(b-a)^{v-j}\mathbb{Y}_{j}^{v}(x;a,b,v)f_{\mathbb{Y}%
,k-j}(x,t;a,b,m-j),
\end{equation*}%
where $f_{\mathbb{Y},k}(x,t;a,b,m)$ and $\mathbb{Y}_{j}^{v}(x;a,b,v)$ are
defined in (\ref{1BB3}) and (\ref{3BE}), respectively.
\end{theorem}

Using definition (\ref{1BB3}) and (\ref{3BE}) in Theorem \ref{TheoremRECUr},
we obtain a recurrence relation for the Bernstein basis functions\ by the
following theorem:

\begin{theorem}
\label{TheoremRECUrTT}Let $a$ and $b$ be nonnegative real parameters with $%
a\neq b$. Let $m$ be a positive integer and let $x\in \left[ a,b\right] $.
Let $k$, $v$ and $n$ be nonnegative integers with $n\geq k$. Then%
\begin{equation*}
\mathbb{Y}_{k}^{n}(x;a,b,m)=\dsum\limits_{j=0}^{v}(b-a)^{v-j}\mathbb{Y}%
_{j}^{v}(x;a,b,v)\mathbb{Y}_{k-j}^{n-v}(x;a,b,m-j).
\end{equation*}
\end{theorem}

\begin{remark}
Substituting $a=0$ and $b=1$ into Theorem \ref{TheoremRECUrTT}, we obtain
the following result:%
\begin{equation*}
B_{k}^{n}(x)=\dsum\limits_{j=0}^{v}B_{j}^{v}(x)B_{k-j}^{n-v}(x).
\end{equation*}
\end{remark}

Substituting $v=1$ into Theorem \ref{TheoremRECUrTT}, we arrive at the
following corollary:

\begin{corollary}
\label{BeT-3}(\textbf{Recurrence Relation}) Let $a$ and $b$ be nonnegative
real parameters with $a\neq b$. Let $m$ be a positive integer and let $x\in %
\left[ a,b\right] $. Let $k$ and $n$ be nonnegative integers with $n\geq k$.
Then%
\begin{eqnarray}
\mathbb{Y}_{k}^{n}(x;a,b,m) &=&\frac{x-a}{b-a}\mathbb{Y}%
_{k-1}^{n-1}(x;a,b,m-1)  \label{2BE} \\
&&+\frac{b-x}{b-a}\mathbb{Y}_{k}^{n-1}(x;a,b,m-1).  \notag
\end{eqnarray}
\end{corollary}

\begin{remark}
Differentiating equation (\ref{1Be}) with respect to $t$, we also get%
\begin{eqnarray*}
&&\frac{x-a}{b-a}f_{\mathbb{Y},k-1}(x,t;a,b,m-1)+\frac{b-x}{b-a}f_{\mathbb{Y}%
,k}(x,t;a,b,m-1) \\
&=&\sum_{n=1}^{\infty }\mathbb{Y}_{k}^{n}(x;a,b,m)\frac{t^{n-1}}{\left(
n-1\right) !}.
\end{eqnarray*}%
From this equation, one can also obtain Corollary \ref{BeT-3}.
\end{remark}

\begin{remark}
By setting $a=0$ and $b=1$ in (\ref{2BE}), one obtains the following
relation:%
\begin{equation*}
B_{k}^{n}(x)=(1-x)B_{k}^{n-1}(x)+xB_{k-1}^{n-1}(x).
\end{equation*}
\end{remark}

\subsection{Multiplication and division by powers of $(\frac{x-a}{b-a})^{d}$
and $(\frac{b-x}{b-a})^{d}$}

In \cite{Goldman}, Buse and Goldman present much background material on
computations with Bernstein polynomials. They provide formulas for
multiplication and division of Bernstein polynomials by powers of $x$ and $%
1-x$ and for degree elevation of Bernstein polynomials. Our method is
similar to that of Buse and Goldman's \cite{Goldman}. In this section we
find two functional equations. Using these equations, we also give new
proofs of both the multiplication and division properties for the Bernstein
polynomials.

By using the generating function in (\ref{1Be}), we provide formulas for
multiplying Bernstein polynomials by powers of $(\frac{x-a}{b-a})^{d}$ and $(%
\frac{b-x}{b-a})^{d}$ and for degree elevation of the Bernstein polynomials.

Using (\ref{1BB3}), we obtain the following \textbf{functional equation}%
\textit{:}%
\begin{equation*}
(\frac{x-a}{b-a})^{d}f_{\mathbb{Y},k}(x,t;a,b,n)=\frac{(k+d)!}{k!t^{d}}f_{%
\mathbb{Y},k}(x,t;a,b,n).
\end{equation*}%
After elementary manipulations in this equation, we get%
\begin{equation}
(\frac{x-a}{b-a})^{d}\mathbb{Y}_{k}^{n}(x;a,b,n)=\frac{n!(k+d)!}{k!(n+d)!}%
\mathbb{Y}_{k+d}^{n+d}(x;a,b,n+d).  \label{1BB4}
\end{equation}%
Substituting $d=1$, we have%
\begin{equation}
(\frac{x-a}{b-a})\mathbb{Y}_{k}^{n}(x;a,b,n)=\frac{k+1}{n+1}\mathbb{Y}%
_{k+1}^{n+1}(x;a,b,n+1).  \label{1BB4a}
\end{equation}

\begin{remark}
Substituting $a=0$ and $b=1$ into (\ref{1BB4a}), we have%
\begin{equation*}
xB_{k}^{n}(x)=\frac{k+1}{n+1}B_{k+1}^{n+1}(x).
\end{equation*}%
The above relation can also be proved by (\ref{1BB3a}) cf. \cite{Goldman}.
\end{remark}

Similarly, using (\ref{3BE}), we obtain%
\begin{equation*}
(\frac{b-x}{b-a})^{d}\mathbb{Y}_{k}^{n}(x;a,b,n)=\frac{n!(n+d-k)!}{\left(
n+d\right) !(n-k)!}\mathbb{Y}_{k}^{n+d}(x;a,b,n+d).
\end{equation*}%
Substituting $d=1$ into the above equation, we have%
\begin{equation}
(\frac{b-x}{b-a})\mathbb{Y}_{k}^{n}(x;a,b,n)=\frac{n+1-k}{n+1}\mathbb{Y}%
_{k}^{n+1}(x;a,b,n+1).  \label{1BB4ab}
\end{equation}%
Consequently, by the same method as in \cite{Goldman}, if we have (\ref{1BB2}%
), then%
\begin{equation}
(\frac{x-a}{b-a})^{d}\mathcal{P}(x,a,b)=\sum_{k=0}^{n}c_{k}^{n}\frac{n!(k+d)!%
}{k!(n+d)!}\mathbb{Y}_{k+d}^{n+d}(x;a,b,n+d),  \label{1BB1a}
\end{equation}%
and%
\begin{equation}
(\frac{b-x}{b-a})^{d}\mathcal{P}(x,a,b)=\sum_{k=0}^{n}c_{k}^{n}\frac{%
n!(n+d-k)!}{\left( n+d\right) !(n-k)!}\mathbb{Y}_{k}^{n+d}(x;a,b,n+d).
\label{1BB1b}
\end{equation}

We now consider \textit{division} properties. We assume that (\ref{1BB2})
holds and that we are given an integer $j>0$. Since $(\frac{x-a}{b-a})^{j}$
divides $\mathbb{Y}_{k}^{n}(x;a,b,n)$ for all $k\geq j$, it follows that $(%
\frac{x-a}{b-a})^{j}$ divides $\mathcal{P}(x,a,b)$. Similarly, using (\ref%
{1BB3}), we obtain the following \textbf{functional equation}\textit{:}%
\begin{equation*}
\frac{f_{\mathbb{Y},k}(x,t;a,b,n)}{(\frac{x-a}{b-a})^{j}}=\frac{(k-f)!t^{j}}{%
k!}f_{\mathbb{Y},k-j}(x,t;a,b,n-j).
\end{equation*}%
For $k\geq j$, from the above equation, we have%
\begin{equation*}
\frac{\mathbb{Y}_{k}^{n}(x;a,b,n)}{(\frac{x-a}{b-a})^{j}}=\frac{n!(k-j)!}{%
k!(n-j)!}\mathbb{Y}_{k-j}^{n-j}(x;a,b,n-j).
\end{equation*}%
By a calculation similar to the calculation in \cite{Goldman}, for $j\leq
n-k $, we have%
\begin{equation*}
\frac{\mathbb{Y}_{k}^{n}(x;a,b,n)}{(\frac{b-x}{b-a})^{j}}=\frac{n!(n-j-k)!}{%
\left( n-k\right) !(n-j)!}\mathbb{Y}_{k}^{n-j}(x;a,b,n-j).
\end{equation*}%
Therefore 
\begin{equation}
\frac{\mathcal{P}(x,a,b)}{(\frac{x-a}{b-a})^{j}}=\sum_{k=j}^{n}c_{k}^{n}%
\frac{n!(k-j)!}{k!(n-j)!}\mathbb{Y}_{k-j}^{n-j}(x;a,b,n-j),  \label{1BB1d}
\end{equation}%
and%
\begin{equation}
\frac{\mathcal{P}(x,a,b)}{(\frac{b-x}{b-a})^{j}}=\sum_{k=0}^{n-j}c_{k}^{n}%
\frac{n!(n-j-k)!}{\left( n-k\right) !(n-j)!}\mathbb{Y}_{k}^{n-j}(x;a,b,n-j).
\label{1BB1e}
\end{equation}

\subsection{Degree elevation}

According to Buse and Goldman \cite{Goldman}, given a polynomial represented
in the univariate Bernstein basis of degree $n$, degree elevation computes
representations of the same polynomial in the univariate Bernstein bases of
degree greater than $n$. Degree elevation allows us to add two or more
Bernstein polynomials which are not represented in the same degree Bernstein
basis functions.

Adding (\ref{1BB4a}) and (\ref{1BB4ab}), we obtain the degree elevation
formula for the Bernstein basis functions:%
\begin{equation*}
\mathbb{Y}_{k}^{n}(x;a,b,n)=\frac{k+1}{n+1}\mathbb{Y}_{k+1}^{n+1}(x;a,b,n+1)+%
\frac{n+1-k}{n+1}\mathbb{Y}_{k}^{n+1}(x;a,b,n+1).
\end{equation*}%
Substituting $d=1$ into (\ref{1BB1b}), and adding these two equations gives
the following degree elevation formula for the Bernstein polynomials:%
\begin{equation}
\mathcal{P}(x,a,b)=\sum_{k=0}^{n}\left( \frac{k}{n+1}c_{k-1}^{n}+\frac{n+1-k%
}{\left( n+1\right) }c_{k}^{n}\right) \mathbb{Y}_{k}^{n+1}(x;a,b,n+1),
\label{1BB1f}
\end{equation}%
where%
\begin{equation*}
c_{k}^{n+1}=\frac{k}{n+1}c_{k-1}^{n}+\frac{n+1-k}{\left( n+1\right) }%
c_{k}^{n}.
\end{equation*}

\begin{remark}
If we set $a=0$ and $b=1$, then (\ref{1BB1f}) reduces to Eq-(2.5) in \cite[%
p. 853]{Goldman}.
\end{remark}

\section{Relation between the generating functions $f_{\mathbb{Y}%
,k}(x,t;a,b,m)$, Poisson distribution and Szasz-Mirakjan type basis functions%
}

The identity of Jetter and St\"{o}ckler represents a pointwise orthogonality
relation for the multivariate Bernstein polynomials on a simplex. This
identity give us a new representation for the dual basis which can be used
to construct general quasi-interpolant operators cf. (See, for details, \cite%
{Jetter2003}, \cite{abel}). As an application of the generating functions
for the basis functions to the identity of Jetter and St\"{o}ckler, Abel and
Li \cite{abel} proved Proposition \ref{Prob. 1}, which is given in this
section. Applying our generating functions to Proposition \ref{Prob. 1}, we
give pointwise orthogonality relations for the Bernstein polynomials and the
Szasz-Mirakjan basis functions.

In this section, we give relations between the Bernstein basis functions,
the binomial distribution and the Poisson distribution. First we we consider
the generalized binomial or Newton distribution (probability function).
Suppose that $0\leq \frac{x-a}{b-a}\leq 1$ and $0\leq \frac{b-x}{b-a}\leq 1$%
. Set%
\begin{equation}
\mathbb{Y}_{k}^{n}(x;a,b,n)=\left( 
\begin{array}{c}
n \\ 
k%
\end{array}%
\right) \left( \frac{x-a}{b-a}\right) ^{k}\left( \frac{b-x}{b-a}\right)
^{n-k}.  \label{1BB1aa}
\end{equation}%
From the above definition, one can see that%
\begin{equation*}
\sum_{k=0}^{n}\mathbb{Y}_{k}^{n}(x;a,b,n)=1.
\end{equation*}

\begin{remark}
If we set $a=0$ and $b=1$, then (\ref{1BB1aa}) reduces to%
\begin{equation*}
\mathbb{Y}_{k}^{n}(x;0,1,n)=\left( 
\begin{array}{c}
n \\ 
k%
\end{array}%
\right) x^{k}(1-x)^{n-k}
\end{equation*}%
which is the binomial or Newton distribution (probabilities) function. If $%
0\leq x\leq 1$ is the probability of an event $E$, then $\mathbb{Y}%
_{k}^{n}(x;0,1,n)$ is the probability that $E$ will occur exactly $k$ times
in $n$ independent trials cf. \cite{Lorentz}.
\end{remark}

Expected value or mean and variance of $\mathbb{Y}_{k}^{n}(x;a,b,n)$ are
given by%
\begin{equation*}
\mu =\sum_{k=0}^{n}k\mathbb{Y}_{k}^{n}(x;a,b,n)=n\left( \frac{x-a}{b-a}%
\right) ,
\end{equation*}%
and%
\begin{equation*}
\sigma ^{2}=\sum_{k=0}^{n}k^{2}\mathbb{Y}_{k}^{n}(x;a,b,n)-\mu ^{2}=\frac{%
n\left( x-a\right) \left( b-x\right) }{\left( b-a\right) ^{2}}.
\end{equation*}%
If we let $n\rightarrow \infty $ in (\ref{1BB1aa}), then we arrive at the
well-known Poisson distribution function:%
\begin{equation}
\mathbb{Y}_{k}^{n}(\frac{b-a}{n}\mu +a;a,b,n)\rightarrow \frac{\mu
^{k}e^{-\mu }}{k!}.  \label{1BB2B}
\end{equation}

The following proposition is proved by Abel and Li \cite[p. 300, Proposition
3]{abel}:

\begin{proposition}
\label{Prob. 1} Let the system $\{f_{n}(x)\}$ of functions be defined by the
generating function%
\begin{equation*}
A_{t}(x)=\dsum\limits_{n=0}^{\infty }f_{n}(x)t^{n}.
\end{equation*}%
If there exists a sequence $w_{k}=w_{k}(x)$ such that%
\begin{equation*}
\dsum\limits_{k=0}^{\infty }w_{k}\mathcal{D}^{k}A_{t}(x)\mathcal{D}%
^{k}A_{z}(x)=A_{tz}(x)
\end{equation*}%
with $\mathcal{D}=\frac{d}{dx}$, then for $i,j=0,1,\cdots $,
\end{proposition}

\begin{equation*}
\dsum\limits_{k=0}^{\infty }w_{k}\mathcal{D}^{k}f_{i}(x)\mathcal{D}%
^{k}f_{j}(x)=\delta _{i,j}f_{i}(x).
\end{equation*}

As an application of Proposition \ref{Prob. 1}, Abel and Li \cite{abel} use
the generating function in Eq-(\ref{1BB1Subg}) for the Bernstein basis
functions. They also use generating functions for the Szasz-Mirakjan basis
functions and Baskakov basis functions.

In this section, we apply our novel generating functions to Proposition \ref%
{Prob. 1}, which give pointwise orthogonality relations for the Bernstein
polynomials and the Szasz-Mirakjan type basis functions, respectively.

As applications of Proposition \ref{Prob. 1}, we give the following examples:

\begin{example}
For given $n$ and $k$, the Bernstein basis functions%
\begin{equation*}
f_{i}(x,n;a,b)=\mathbb{Y}_{i}^{n}(x;a,b,n)=\left( 
\begin{array}{c}
n \\ 
i%
\end{array}%
\right) \left( \frac{x-a}{b-a}\right) ^{k}(\frac{b-x}{b-a})^{n-k}
\end{equation*}%
are generated by the function in (\ref{1BB3}), that is%
\begin{equation*}
A_{t}(x)=\frac{t^{k}\left( x-a\right) ^{k}e^{(b-x)t}}{(b-a)^{n}k!}%
=\dsum\limits_{i=0}^{\infty }\frac{f_{i}(x,n;a,b)}{i!}t^{i}.
\end{equation*}%
It is easy to check that Proposition \ref{Prob. 1} holds with $%
w_{k}=w_{k}(x)=\mathbb{Y}_{k}^{n}(x;a,b,n)$.
\end{example}

\begin{example}
Using (\ref{1BB2B}), for $i\geq 0$, we generalize the \textit{Szasz-Mirakjan
type basis functions} as follows%
\begin{equation*}
f_{i}(x,n;a,b)=\frac{(n\frac{x-a}{b-a})^{i}e^{-n\frac{x-a}{b-a}}}{i!},
\end{equation*}%
where $a$ and $b$ are nonnegative real parameters with $a\neq b$, $n$ is a
positive integer and $x\in \left[ a,b\right] $. The functions $%
f_{i}(x,n;a,b) $ are generated by%
\begin{equation*}
A_{t}(x)=\exp \left( (t-1)n\left( \frac{x-a}{b-a}\right) \right)
=\dsum\limits_{i=0}^{\infty }f_{i}(x,n;a,b)t^{i},
\end{equation*}%
where $\exp (x)=e^{x}$. In this case, Proposition \ref{Prob. 1} holds with $%
w_{k}=w_{k}(x)=\frac{\left( \frac{x-a}{b-a}\right) ^{k}}{n^{k}k!}$.
Therefore, we have%
\begin{equation*}
\dsum\limits_{k=0}^{\infty }\frac{\left( \frac{x-a}{b-a}\right) ^{k}}{n^{k}k!%
}\mathcal{D}^{k}f_{i}(x,n;a,b)\mathcal{D}^{k}f_{i}(x,n;a,b)=\delta
_{i,j}f_{i}(x,n;a,b).
\end{equation*}
\end{example}

\begin{remark}
If $a=0$ and $b=1$ in Example 2, then we arrive at the \textit{%
Szasz-Mirakjan basis functions which are given in} \cite[p. 300, Example 2]%
{abel}.
\end{remark}

\begin{acknowledgement}
The author would like to thank Professor Ronald Goldman (Rice University,
Houston, USA) for his very valuable comments, criticisms and for his very
useful suggestions on this present paper.

The present investigation was supported by the \textit{Scientific Research
Project Administration of Akdeniz University.}
\end{acknowledgement}

\end{document}